\documentclass[a4paper,12pt]{amsart}

\usepackage[utf8]{inputenc}
\usepackage[english]{babel}

\usepackage{amsmath,amssymb,amsthm}
\usepackage{graphicx}
\usepackage[alphabetic]{amsrefs}
\usepackage{tikz-cd}
\usepackage{pst-all}

\usepackage{etoolbox}
\apptocmd{\sloppy}{\hbadness 10000\relax}{}{}

\newtheorem{theorem}{Theorem}[section]
\newtheorem{lemma}[theorem]{Lemma}

\newtheorem{question}[theorem]{Question}
\newtheorem{proposition}[theorem]{Proposition}
\newtheorem{corollary}[theorem]{Corollary}
\theoremstyle{definition}
\newtheorem{definition}[theorem]{Definition}
\newtheorem{remark}[theorem]{Remark}
\newtheorem{example}[theorem]{Example}

\newcommand{\NN}{\mathbb{N}}

\newcommand{\RR}{\mathbb{R}}
\newcommand{\CC}{\mathbb{C}}
\DeclareMathOperator{\id}{id}
\DeclareMathOperator{\lie}{Lie}
\newcommand{\po}[2]{\left\{ {#1}, {#2} \right\}}
\newcommand{\cont}{\mathcal{C}}
\newcommand{\holo}{\mathcal{O}}
\newcommand{\Kappa}{\mathrm{K}}

\title{Holomorphic automorphisms of Markov-type surfaces}
\author{Rafael B. Andrist}


\begin{document}

\begin{abstract}
Every complex surface of Markov type, i.e.\ the variety given by $x^2 + y^2 + z^2 + Exyz - Ax - By - Cz  - D = 0$, has the symplectic density property and the Hamiltonian density property. We prove a singular symplectic version of the Anders{\'e}n--Lempert theorem for normal reduced affine complex varieties and apply it to describe the holomorphic symplectic automorphisms of a complex surface of Markov type. To this end, we also investigate the germs of vector fields in isolated singularities of type $A_k$ and $D_k$. Moreover, we show that any injective self-map of the set of ordered Markov triples can be realized by a holomorphic symplectic automorphism. 	
\end{abstract}

\maketitle

\section{Introduction}

The Diophantine solutions of \emph{Markov's equation} 
\[
x^2 + y^2 + z^2 = 3xyz
\]
were originally considered by Markov \cites{MR4788527, MR1510073} in 1879--1880. Due to symmetry, any permutation of the coordinates $(x,y,z)$ of a solution yields another solution. Moreover, $(x,y,z) \mapsto (x,y,3xy-z)$ is an involution that preserves the equation. In fact, Markov already proved that every Diophantine solution of the equation can be obtained from $(x,y,z) = (1,1,1)$ using only these maps. These solutions $(x,y,z) \in \NN^3$, e.g.\ $(1,1,1), (1,1,2), (1,2,5), (1,5,13), (2,5,29), \dots$ are now called \emph{Markov triples}, and any of such $x,y,z$ are called \emph{Markov numbers}. In 1913, Frobenius \cite{Frobenius} raised the question whether a Diophantine solution of Markov's equation with $x \leq y \leq z$ is uniquely determined by $z$, which is now known as the \emph{uniqueness conjecture} for Markov triples.

It is natural to study the complex solutions of the Markov equation which then define a singular affine variety, called the \emph{Markov surface}. 

The Markov surface admits a projective completion where the dual graph of the boundary divisor is circular with weights $((0,0))$, see Perepechko \cite{MR4337481}. Thus, it is a so-called \emph{generalized Gizatullin surface} according to Kaliman, Kutzschebauch and Leuenberger \cite{MR4083242}*{Theorem 1.4, Case (2a)}, i.e.\ the subgroup of its holomorphic automorphisms generated by complete algebraic vector fields acts with an open orbit.

In this article, we consider the slightly more general surfaces of \emph{Markov type}, i.e.\
\[
M := \{ (x,y,z) \in \CC^3 \,:\, x^2 + y^2 + z^2 + Exyz - Ax - By - Cz  - D = 0 \}
\]
where $E \neq 0$. Markov's equation would correspond to the special case $A=B=C=D=0$ and $E=-3$. In algebraic geometry, it is common to rescale to $E = 1$. We find it more convenient to keep the constant $E \neq 0$ for most of our computations. This equation also comprises an affine form of the \emph{Cayley cubic} $x^2 + y^2 + z^2 + xyz = 4$, see \cite{MR2649343}*{Section 1.5} as well as Mordell's equation \cite{MR0056619} $x^2 + y^2 + z^2 = 2xyz + n$. The case $E = 0$ is not considered here; this would simply be a quadric surface whose holomorphic automorphisms have been described both in the smooth and singular case (for the latter, see Kutzschebauch, Leuenberger and Liendo \cite{MR3320241}).

The group of algebraic automorphisms of $M$ is well studied, see El-Huti \cite{MR0342518}, Cantat and Loray \cite{MR2649343}*{Theorem 3.1} and Perepechko \cite{MR4337481}*{Theorem 3.3}. In particular, the algebraic automorphisms of $M$ form a infinite discrete group.

On the regular part of the surface $M$, there exists a natural algebraic volume form that is induced by the volume form of the ambient space:
\begin{equation*}
\omega = \frac{dx \wedge dy}{2z + Exy - C} = \frac{dy \wedge dz}{2x + Eyz - A} = \frac{dz \wedge dx}{2y + Exz - B}
\end{equation*}
Every algebraic automorphism of $M$ either preserves or anti-preserves $\omega$, see	 \cite{MR2649343}*{Proposition 3.6}. Since we are considering surfaces, it will be convenient to consider the volume form also as symplectic form, and we will switch between these two viewpoints whenever it is convenient.

While the algebraic automorphisms of the surfaces of Markov type have been studied, its transcendental automorphisms have so far not been considered at all. In this article, we find a large number of holomorphic automorphisms of $M$. In fact, we prove the so-called Hamiltonian density property and the symplectic density property (see Section \ref{sec-DP}) for $M$. As a consequence, $M$ allows for a Runge-type approximation theorem for holomorphic symplectic injections by holomorphic symplectic automorphisms, and the holomorphic automorphisms of $M$ form an infinite-dimensional group. This shows a strong contrast between the zero-dimensional algebraic automorphism group and the infinite-dimensional holomorphic automorphism group.

\begin{theorem}
\label{thm-main-1}
The Markov-type surface
\[
M = \{ (x,y,z) \in \CC^3 \,:\, x^2 + y^2 + z^2 + Exyz - Ax - By - Cz  - D = 0 \}
\]
with the natural symplectic form on its regular part has the Hamiltonian density property.
\end{theorem}

\begin{theorem}
\label{thm-main-2}
The Markov-type surface
\[
M = \{ (x,y,z) \in \CC^3 \,:\, x^2 + y^2 + z^2 + Exyz - Ax - By - Cz  - D = 0 \}
\]
with the natural symplectic form on its regular part has symplectic density property.
\end{theorem}

\begin{corollary}[see Corollary \ref{cor-inftrans}]
The group of holomorphic symplectic automorphisms acts $m$-transitively on the regular part of $M$ for every $m \in \NN$.
\end{corollary}

\begin{corollary}
The group of holomorphic symplectic automorphisms of $M$ is infinite-dimensional.
\end{corollary}

\begin{theorem}
\label{thm-groupdesc}
The group generated by the complete flows of the vector fields corresponding to the Hamiltonian functions $x^k, y^k, z^k$ $(k \in \NN)$ is dense (in the topology of locally uniform convergence) in the identity path-component of the group of holomorphic symplectic automorphisms.
\end{theorem}

The explicit form of the $1$-parameter family of automorphisms corresponding to these Hamiltonian functions $x^k, y^k, z^k$ is computed in Lemma \ref{lemma-completeness}.

\bigskip

We now specialize again to the original Markov surface. By definition, a closed discrete subset $A$ is called \emph{tame} if every injective self-map $A \to A$ can be interpolated by a holomorphic automorphism, see Section \ref{sec-tame} for more details.

\begin{theorem}
\label{theorem-tame}
The ordered Markov triples form a tame subset of the Markov surface $x^2 + y^2 + z^2 = 3xyz$ w.r.t.\ its group of holomorphic symplectic automorphisms. 
\end{theorem}

\begin{corollary}
The Markov surface contains \emph{Fatou--Bieberbach domains}, i.e.\ there exists a biholomorphic map $F \colon M \to \Omega$ where $\Omega \subsetneq M$ is an open subset. 
\end{corollary}
\begin{proof}
This follows from the existence of tame and non-tame sets and a well-known construction of Forstneri\v{c}, see e.g.\ the proof of \cite{MR3700709}*{Theorem 4.17.1}.
\end{proof}

\begin{corollary}
Given any enumeration of ordered Markov triples, there exists a holomorphic symplectic automorphism of the Markov surface that maps each ordered Markov triple to its successor.
\end{corollary}
\begin{proof}
This is a direct consequence of the definition of a tame set. 
\end{proof}

We conclude with two open questions:

\begin{question}
Can the dynamical properties of the holomorphic automorphism of the preceding corollary be related to the growth estimates of Markov numbers?
\end{question}

The growth of the Markov numbers is conjectured by Zagier \cite{MR0669663} to be $m_n = \frac{1}{3} \exp(C \sqrt{n} + o(1))$ with an exlicitly computable constant $C$ where $m_n$ stands for the $n$-th Markov number, enumerated in increasing order.

\begin{question} \hfill
\begin{enumerate} 
\item Are all holomorphic automorphisms of $M$ either symplectic or anti-symplectic? 
\item Can we describe all (path-)connected components of the group of holomorphic automorphisms?
\end{enumerate}
\end{question}

Note that this question about the group of holomorphic automorphisms is also still open for much ``easier'' examples such as $\CC^\ast \times  \CC^\ast$, see e.g.\ the recent survey by Forstneri\v{c} and Kutzschebauch \cite{MR4440754}.

\bigskip

This paper is organized as follows: In Section \ref{sec-newflows} we introduce complete algebraic vector fields with holomorphic flows, i.e.\ they generate $1$-parameter groups of holomorphic automorphisms. In Section \ref{sec-DP} we provide the theoretical framework for the Hamiltonian and symplectic density property, adjusted to the singular situation. In Section \ref{sec-singular} we determine and analyze the isolated singularities that occur for Markov type surfaces. We carry out these computations for all isolated singularities of type $A_k$ and $D_k$ since it might be of independent interest. In Section \ref{sec-main} we prove the density properties for the Markov-type surfaces. Finally, Section \ref{sec-tame} deals with the construction of tame sets. 

\section{New Complete Vector Fields}
\label{sec-newflows}
A key ingredient is the following Lemma which produces new $1$-parameter families of holomorphic automorphisms on the Markov-type surfaces. 
\begin{lemma}
\label{lemma-completeness}
All of the three tangential vector fields $V^x, V^y, V^z$ on $M$ are complete:
\[
\begin{alignedat}{4}
V^x &:= & & \phantom{+} (2z + Exy - C) \frac{\partial}{\partial y}& - (2y + Exz - B) \frac{\partial}{\partial z}& \\
V^y &:= & - (2z + Exy - C) \frac{\partial}{\partial x}& & + (2x + Eyz - A) \frac{\partial}{\partial z}& \\
V^z &:= & (2y + Exz - B) \frac{\partial}{\partial x}& - (2x + Eyz - A) \frac{\partial}{\partial y}& 
\end{alignedat}
\]
\end{lemma}
\begin{proof}
The tangent space to $M$ in $(x,y,z) \in M$ is the kernel of:
\[
N = (2x + Eyz - A) dx + (2y + Exz - B) dy + (2z + Exy - C) dz
\]
It is now straightforward that $V^x$, $V^y$ and $V^z$ are tangential vector fields.

The defining equation of $M$ is symmetric with respect to simultaneous permutations of the variables 
$(x,y,z)$ and the coefficients $(A,B,C)$, hence we only need to consider the vector field $V^z$ and its flow \[\varphi^z_t(x,y,z) = (x(t), y(t), z(t)).\]

Since $z(t) = z$ is constant in time, we are led to a linear system of ODEs in $(x,y)$ with constant coefficients, depending on a parameter $z$:
\[
\begin{pmatrix}
\dot{x} \\ \dot{y}
\end{pmatrix}
=
\underbrace{
\begin{pmatrix}
zE & 2 \\
-2 & -Ez
\end{pmatrix}}_{=: Q(z)}
\cdot
\begin{pmatrix}
x \\ y
\end{pmatrix}
+
\begin{pmatrix}
-B \\ A
\end{pmatrix}
\]
The characteristic polynomial of $Q(z)$ is $\lambda \mapsto \lambda^2 + 4 - E^2 z^2$.
For the inhomogeneous system we then obtain by a standard calculation
\begin{align*}
x(t) &= \left(x - \frac{2 A - B E z}{\gamma^2(z)} \right) \cos (\gamma(z) t) \\ &\quad + (2y + Exz - B) \frac{\sin (\gamma(z) t)}{\gamma(z)} + \frac{2 A - B E z}{\gamma^2(z)} \\
y(t) &= (2x + Eyz - A)\frac{\sin (\gamma(z) t)}{-\gamma(z)} \\ &\quad + \left(y - \frac{2 B - A E z}{\gamma^2(z)}\right) \cos (\gamma(z) t) + \frac{2 B - A E z}{\gamma^2(z)}
\end{align*}
where $\gamma(z) := \sqrt{4 - E^2 z^2}$, i.e.\ $\gamma^2(z) = - \lambda^2$. Any ambiguities introduced by the complex square root cancel out, as can be seen by the power series development of $\sin$ and $\cos$. The singularities introduced by $\gamma^2(z)$ in the denominator cancel out as well.
\end{proof}

\begin{remark}
The complex root $\gamma(z) = \sqrt{4 - E^2 z^2}$ appears in the proof naturally since $4 - E^2 z^2$ is the determinant of the matrix describing the homogeneous part of the system of ODEs that we solve. For the Markov surface, i.e.\ with $E = -3$, this term also shows up in the original comparison of the Markov spectrum and the Lagrange spectrum by Markov \cite{MR4788527}:
\[
\mathcal{L}_{< 3} = \left\{ \frac{\sqrt{9 z^2 - 4}}{z} \,:\, z \in \mathcal{M} \right\}
\] 
where $\mathcal{M}$ denotes the set of Markov numbers, and $\mathcal{L}_{< 3}$ the Langrange spectrum below $3$. The notation used here follows the textbook of Aigner \cite{MR3098784}*{Section 2.2}.
\end{remark}

\begin{corollary}
Let $f$ be an entire holomorphic function of one variable. Then the vector fields
\[
f(x) V^x, \quad f(y) V^y, \quad f(z) V^z
\]
are complete.
\end{corollary}
\begin{proof}
By symmetry, it is enough to consider the vector field $f(z) V^z$. Since $V^z(z) = 0$ we have $z \mapsto f(z) \in \ker V^z$. Therefore the shear vector field $f(z) V^z$ is complete as well. Its flow is given by $\varphi^z_{f(z) t}$. 
\end{proof}

\begin{remark}
\label{rem-spanning}
The vector fields $V^x, V^y, V^z$ span the tangent space in every smooth point of $M$: Arranging the coefficients of the respective vector fields in a matrix,
\[
\begin{pmatrix}
0 & 2z + Exy - C & - (2y + Exz - B) \\
- (2z + Exy - C) & 0 & 2x + Eyz - A \\
2y + Exz - B & -(2x + Eyz - A) & 0
\end{pmatrix} \cdot
\begin{pmatrix}
\displaystyle \frac{\partial}{\partial x} \\[12pt] \displaystyle \frac{\partial}{\partial y} \\[12pt] \displaystyle \frac{\partial}{\partial z} 
\end{pmatrix}
\]
we see that the only points where they do not span the tangent space, must satisfy:
\[
2x + Eyz - A = 0, \; 2y + Exz - B = 0, \; 2z + Exy - C = 0 
\]
However, these are also the coefficients of the $1$-form $N$ whose kernel is the tangent space.

\end{remark}

\section{Hamiltonian and Symplectic Density Property}
\label{sec-DP}

Recall that a holomorphic vector field $V$ is called \emph{symplectic} if the Lie derivative of the symplectic form $\omega$ w.r.t.\ $V$ vanishes. 
By Cartan's homotopy formula \( 0 = \mathcal{L}_V \omega = d i_V \omega + i_V d \omega \), this is equivalent to saying that the form $i_V \omega$ is closed since $\omega$ is closed by assumption. A holomorphic vector field $V$ is called \emph{Hamiltonian} if $i_V \omega$ is exact, i.e.\ if there exists a holomorphic function $h$ -- called the \emph{Hamiltonian (function)} of $V$ -- such that $i_V \omega = dh$. If the holomorphic de Rham cohomology of the space vanishes, then every symplectic vector field is Hamiltonian.

The Hamiltonian vector fields naturally form a Lie algebra with the commutator $[.,.]$ as the Lie bracket. The Hamiltonian function of a commutator is computed as the \emph{Poisson bracket} $\{.,.\}$ of the corresponding Hamiltonian functions:
Assume that $i_V \omega = df$ and $i_W \omega = dg$ for functions $f$ and $g$. Then we have that
\[
i_{[V, W]} \omega = -d \{ f, g \}
\]

Since the Poisson bracket satisfies the Leibniz rule, it is completely determined by its values on the functions that generate the ring of functions. The Poisson bracket can also be computed directly from the symplectic form:
\begin{equation}
\label{eq-poisson-general}
\{ f, g \} = \omega(V, W)
\end{equation}

\bigskip

The Hamiltonian density property for Stein manifolds was recently introduced by Gaofeng Huang and the author \cite{CaloSymplo}. The notion of the Hamiltonian density property is itself a variation of volume density property and of the density property introduced by Varolin \cite{MR1829353}. For an overview of this theory, we refer to the recent survey of Forstneri\v{c} and Kutzschebauch \cite{MR4440754}. The singular case for the density property was treated in a paper by Kutzschebauch, Leuenberger and Liendo \cite{MR3320241}. However, singular versions for the volume density property or the symplectic/Hamiltonian density property do not yet exist.

\medskip

Since a symplectic form is only well-defined on the regular part of a complex space, also the notion of a Hamiltonian or symplectic vector fields is a priori only defined on the regular part. However, the Hamiltonian functions will extend holomorphically to the singularity set if it has at least codimension $2$ due to the Hartogs phenomenon and the Riemann removable singularity theorem.

\begin{definition}
Let $X$ be a normal reduced complex space equipped with a holomorphic symplectic form $\omega$ on its regular part $X_{\mathrm{reg}}$. Let $V$ be vector field on $X$, i.e.\ a holomorphic global section of its tangent sheaf.
\begin{enumerate}
\item $V$ is called \emph{symplectic} if $i_V \omega$ is a closed form on $X_{\mathrm{reg}}$. 
\item $V$ is called \emph{Hamiltonian} if $i_V \omega$ is an exact form on $X_{\mathrm{reg}}$.  
\end{enumerate}
\end{definition}

In general, can't expect that $i_V \omega$ extends holomorphically to the singular part of $X$.

\begin{definition}
Let $X$ be a normal reduced Stein space equipped with a holomorphic symplectic form $\omega$ on its regular part $X_{\mathrm{reg}}$.
\begin{enumerate}
\item We say that $(X, \omega)$ has the \emph{symplectic density property} if the Lie algebra generated by the complete holomorphic symplectic vector fields on $X$ vanishing on $X_{\mathrm{sing}}$ is dense (w.r.t.\ locally uniform convergence) in the Lie algebra of all holomorphic symplectic vector fields on $X$ vanishing on $X_{\mathrm{sing}}$.
\item We say that $(X, \omega)$ has the \emph{relative symplectic density property of order $\ell \in \NN$ w.r.t.\ an analytic subset $A \subset X$} if the Lie algebra generated by the complete holomorphic symplectic vector fields on $X$ vanishing on $A$ to order $\geq \ell$ is dense (w.r.t.\ locally uniform convergence) in the Lie algebra of all holomorphic symplectic vector fields on $X$ vanishing on $A$ to order $\geq \ell$.
\item We say that $(X, \omega)$ has the \emph{Hamiltonian density property} if the Lie algebra generated by the complete holomorphic Hamiltonian vector fields on $X$ vanishing on $X_{\mathrm{sing}}$ is dense (w.r.t.\ locally uniform convergence) in the Lie algebra of all holomorphic Hamiltonian vector fields on $X$ vanishing on $X_{\mathrm{sing}}$.
\end{enumerate}
\end{definition}

As a preparation to deal with the singularities, we first need three technical lemmas:

\begin{lemma}
\label{lemma-spanning}
Let $X$ be a normal reduced Stein space equipped with a holomorphic symplectic form $\omega$ on its regular part $X_{\mathrm{reg}}$. Then there exist finitely many Hamiltonian vector fields $\Theta^1, \dots, \Theta^N$ on $X$ such that for all $p \in X_{\mathrm{reg}}$ we have $\mathrm{span} \{ \Theta^1_p, \dots, \Theta^N_p \} = T_p X$. 
\end{lemma}
\begin{proof}
By the holomorphic analog of the Darboux theorem, see e.g.\ \cite{MR3705282}*{Theorem A.1}, for every point $p \in X_{\mathrm{reg}}$ there exist local holomorphic coordinates $(z_1, \dots, z_n, w_1, \dots, w_n)$ around $p$ such that $\omega = dz_1 \wedge dw_1 + \dots + dz_n \wedge dw_n$. Thus, the coordinate vector fields $\frac{\partial}{\partial z_1}, \dots, \frac{\partial}{\partial z_n}, \frac{\partial}{\partial w_1}, \dots, \frac{\partial}{\partial w_n}$ span $T_p X$. The corresponding Hamiltonian functions are $f^1_p = w_1, \dots, f^n_p = w_n, f^{n+1}_p = -z_1, \dots, f^{2n}_p = -z_n$, respectively. For every connected component $X_j$ of $X$ we choose such a point $p_j \in X_{j,\mathrm{reg}}$. Since $X$ is Stein, we can interpolate up to the first derivative each of the Hamiltonian functions $f^k_{p_j}$ in all points $p_j$ by a holomorphic function $F^k \colon X \to \CC$. The vector fields corresponding to $F_k$ will still be spanning $T_{p_j} X$ for all $p_j$ and thus be spanning in Zariski-open subsets of each $X_j$. We iterate this procedure on the analytic subvariety $A$ where the vector fields are not spanning. Since $A$ is at least of codimension $1$, this procedure terminates after finitely many steps. 
\end{proof}

\begin{example}
For the Markov-type surface $M$ we can choose the Hamiltonian vector fields $V^x, V^y, V^z$ that satisfy
$i_{V^x} \omega = d(-x)$ etc., see Remark \ref{rem-spanning}.
\end{example}

\begin{lemma}
\label{lemma-stalks}
Let $X$ be a normal reduced affine complex variety equipped with a holomorphic symplectic form $\omega$ on its regular part $X_{\mathrm{reg}}$. Then there exists $\ell \in \NN$ and there exist finitely many Hamiltonian vector fields $\Theta^1, \dots, \Theta^N$ such that every holomorphic vector field $\Xi$ on $X$ that vanishes on $X_{\mathrm{sing}}$ to at least order $\ell$ can be written as
\begin{equation}
\label{eq-globalsection}
\Xi = \sum_{k = 1}^N g_k \Theta^k
\end{equation}
for some holomorphic functions $g_k \colon X \to \CC$.
\end{lemma}
\begin{proof}
By Lemma \ref{lemma-spanning} there exist finitely many Hamiltonian vector fields $\Theta^1, \dots, \Theta^N$ such that for all $p \in X_{\mathrm{reg}}$ we have $\mathrm{span} \{ \Theta^1_p, \dots, \Theta^N_p \} = T_p X$. By an application of the Nakayama lemma, their germs actually generate the stalk $\mathcal{T}_p$ of the tangent sheaf of $X$. In the singularity set $X_{\mathrm{sing}}$, these vector fields may not generate the germs. But since $X$ is algebraic, $X_{\mathrm{sing}}$ has only finitely many components, and $\Theta^1_p, \dots, \Theta^N_p$ generate the germs of all vector fields that vanish to a high enough order $\geq \ell$ in $X_{\mathrm{sing}}$. Let $\mathcal{S}$ denote the subsheaf of $\mathcal{T}$ vanishing in the set $X_{\mathrm{sing}}$ to at least order $\ell$. The sheaf $\mathcal{S}$ is coherent, and by an application of Cartan's Theorem B, every global section of $\mathcal{S}$ is of the form \eqref{eq-globalsection}.
\end{proof}

\begin{lemma}
\label{lemma-deRhamtrivial}
Let $X$ be a normal reduced affine complex variety with a holomorphic symplectic form $\omega$ on its regular part $X_{\mathrm{reg}}$. Let $\Omega \subseteq X$ be an open subset (in the Euclidean topology) with trivial first holomorphic de Rham cohomology group $H^1_{d}(\Omega) = 0$.
Then every holomorphic symplectic vector field on $\Omega$ that vanishes in $\Omega_{\mathrm{sing}}$ to order $\geq \ell$ (same as in Lemma \ref{lemma-stalks}) is a Hamiltonian vector field on $\Omega$.
\end{lemma}
\begin{proof}
Let $\Xi$ be a holomorphic vector field on $\Omega$. If $\Xi$ is symplectic, then by definition $i_\Xi \omega$ is a closed form on $\Omega_{\mathrm{reg}}$. 
If $\Xi$ vanishes in $\Omega_{\mathrm{sing}}$ to order $\geq \ell$, then by Lemma \ref{lemma-stalks} we can write 
\begin{equation}
\Xi = \sum_{k = 1}^N g_k \Theta^k
\end{equation}
where $g_k \colon \Omega \to \CC$ are holomorphic coefficient functions and $i_{\Theta_k} \omega = d f_k$ for some holomorphic functions $f_k \colon X \to \CC$, i.e.\ the forms $i_{\Theta_k} \omega$ extend to the singularities $X_{\mathrm{sing}}$. 
Combining these results, we obtain on $\Omega$ that
\[
i_\Xi \omega = \sum_{k = 1}^N g_k i_{\Theta^k} \omega = \sum_{k = 1}^N g_k d f_k,
\]
i.e.\ all the involved terms extend to the singularities. 
Since we require the first holomorphic de Rham cohomology on $\Omega$ to vanish, we can now write
\[
i_\Xi \omega = d f
\]
for some holomorphic function $f \colon \Omega \to \CC$. 	
\end{proof}

\begin{corollary}
Let $X$ be a normal reduced affine complex variety with with a holomorphic symplectic form $\omega$ on its regular part $X_{\mathrm{reg}}$. Assume that $H^1_{d}(X) = 0$. If $X$ has the Hamiltonian density property, then there exists $\ell \in \NN$ such that $X$ has the relative symplectic density property of order $\ell$ w.r.t.\ $X_{\mathrm{sing}}$.
\end{corollary}

In the next section, we will show that this corollary can be improved for Markov surfaces.

\begin{theorem}
\label{thm-AL}
Let $X$ be a normal reduced affine complex variety with a holomorphic symplectic form $\omega$ on its regular part $X_{\mathrm{reg}}$. Assume that $(X, \omega)$ has the relative symplectic density property of order $\ell$ w.r.t.\ $X_{\mathrm{sing}}$.

Let $\Omega \subseteq X$ be a open subset with trivial first holomorphic de Rham cohomology group $H^1_{d}(\Omega) = 0$. Let $\Omega \times [0, 1] \ni (x,t) \mapsto \varphi_t(x) \in X$ be a jointly $\cont^1$-smooth map such that the following holds:
\begin{enumerate}
\item $\varphi_0 \colon \Omega \to X$ is the natural embedding.
\item $\varphi_t \colon \Omega \to X$ is a holomorphic symplectic injection for each $t \in [0,1]$.
\item $\varphi_t(\Omega)$ is a Runge subset of $X$ for each $t \in [0,1]$.
\item $\varphi_t|X_{\mathrm{sing}} = \id_{X_{\mathrm{sing}}}$ up to order $\ell$
\end{enumerate}
Then for every compact $K \subset \Omega$, every $\varepsilon > 0$ and every choice of metric on $X$ that induces its topology, there exists
a continuous family $\Phi_t \colon X \to X$ of holomorphic symplectic automorphisms  such that
\[
\sup_{x \in K} d(\varphi_1(x), \Phi_1(x)) < \varepsilon
\]
Moreover, $\Phi_1$ can be written as a finite composition of flows of complete vector fields that are generators of the Lie algebra of holomorphic symplectic vector fields on $X$.
\end{theorem}

\begin{proof}
Since $X$ is normal, $X_{\mathrm{sing}}$ is of codimension at least $2$. We follow the proof for the symplectic density property in the smooth case in \cite{CaloSymplo}*{Theorem 12} which in turn follows closely the proof for the volume density property \cite{MR1829353}. We only indicate how to adapt the proof to the singular situation. 

The general idea is a follows: $\frac{d \varphi_t}{d t} \circ \varphi^{-1}_t$ defines a time-dependent holomorphic vector field on $\Omega$. We partition $[0,1]$ into small intervals on which we approximate this time-dependent vector field with vector fields that are no longer time-dependent. It is easy to see that if $\varphi_t$ is symplectic, then so are all the involved vector fields. The only remaining step is to extend each of these vector fields holomorphically and symplectically from $\Omega$ to $X$:

By Lemma \ref{lemma-deRhamtrivial} every symplectic holomorphic vector field $\Xi$ on $\Omega$ is a Hamiltonian vector field on $\Omega$, i.e.\ $i_{\Xi} \omega = df$ for some holomorphic $f \colon \Omega \to \CC$. We use Runge approximation for $f$ on $\Omega$ to obtain a globally defined Hamiltonian function $\widetilde{f}$ and hence the corresponding vector field $\widetilde{\Xi}$ on $X_{\mathrm{reg}}$. By the Hamiltonian density property, $\widetilde{f}$ can be approximated by Poisson--Lie combinations of the Hamiltonian functions corresponding to complete vector fields that are defined on $X$. Therefore, $\widetilde{\Xi}$ can be approximated by the corresponding Lie combination of these complete vector fields. By assumption, these vector fields, hence also their Lie combinations, are already defined on $X$. Approximation on $\Omega_{\mathrm{sing}}$ follows from the maximum principle.  
\end{proof}

\begin{corollary}
\label{cor-inftrans}
Let $X$ be a normal reduced Stein space with the Hamiltonian density property and with $\dim_\CC X \geq 2$. Then the group of holomorphic symplectic automorphisms acts $m$-transitively on the regular part of $X$ for every $m \in \NN$.
\end{corollary}

\begin{proof}
The proof is identical to the volume-preserving case or to the Hamiltonian case in the smooth setting, since the points lie in the regular part and their Runge neighborhoods can be chosen biholomorphic to disjoint unions of small enough balls in the regular part as well.
\end{proof}

\section{Singularities}
\label{sec-singular}

The singularities of normal cubic surfaces are well-known since the studies of Cayley and Schläfli. Except for the simple elliptic case $x^3 + y^3 + z^3 + 3\lambda x y z = 0, \lambda^3 \neq 1,$ they are all isolated ADE singularities. However, we will need more precise information which singularities occur for the surfaces of Markov type and how the germs of the vector fields in these singularities look like. 

\begin{remark}
\label{rem-Poincare}
By a result of Reiffen \cite{MR0223599}, the Poincar{\'e} Lemma holds in the singular case if a certain local holomorphic contractibility condition is satisfied. For hypersurfaces, this condition is equivalent to the germ of the hypersurface being given by a so-called quasi-homogeneous polynomial thanks to a result of Saito \cite{MR0294699}. By Proposition \ref{prop-Markov-singularities} below, $M$ only has only isolated singularities of type $A_k$ and $D_4$. ADE singularities are always quasi-homogeneous.
\end{remark}

\begin{proposition}
\label{prop-Markov-singularities}
Let $M$ be a surface of Markov type with $E = 1$. 

\begin{enumerate}
\item If $A = \pm 8$, $B = \pm 8$ and $C = \pm 8$ with an even number of negative signs, and $D = -28$ if all signs are positive or $D = 36$ if two signs are negative, then $M$ has an isolated singularity of type $D_4$.
\item Otherwise, $M$ is smooth or has only isolated singularities of type $A_k$, $k \leq 5$. In particular, if $A = B = C = D = 0$, then $M$ has one isolated singularity which is of type $A_1$. 
\end{enumerate}
\end{proposition}

\begin{proof}
Recall that
\[
N = (2x + yz - A) dx + (2y + xz - B) dy + (2z + 	xy - C) dz
\]
First, we need to find the Morse index of the critical points, i.e.\ when $N=0$. 
The Hessian of the defining equation in the point $(x,y,z) \in M$ is given by 
\[
H = \begin{pmatrix}
2 & z & y \\
z & 2 & x \\
y & x & 2 
\end{pmatrix}
\]
If $H$ has rank $1$, then necessarily $4 - x^2 = 0$, $4 - y^2 = 0$ and $4 - z^2 = 0$. This leaves only the choices $x = \pm 2$, $y = \pm 2$ and $z = \pm 2$. Moreover, we also have $xz - 2y = 0$ which implies that the number of negative signs must be even. Together, these four equations also imply that $xy - 2z = 0$ and $yz - 2x = 0$. In a critical point, $xz - 2y = 0$ and $2y + xz - B = 0$ require that $4y = B$, i.e.\ $B = \pm 8$. Similarly, it follows that $A = \pm 8$ and $C = \pm 8$ as well. By a coordinate transformation $x \mapsto \pm x$, $y \mapsto \pm y$ and $z \mapsto \pm z$ with an even number of negative signs, we obtain the equation
\[
x^2 + y^2 + z^2 + xyz - 8x - 8y - 8z - D = 0 
\]
This has a solution for $(x,y,z) = (2,2,2)$ and $D = -28$. If we flip two signs of $x$, $y$, $z$, then we need to choose $D = 36$ instead. By Lemma \ref{lemma-Markov-D4} below, this surface has one isolated singularity of type $D_4$.

If $H$ has rank $2$, i.e.\ co-rank $1$, then the normal form of the singularity is given by $x^2 + y^2 + z^{k+1}$ with $k \geq 1$, which is an $A_k$-singularity, see e.g.\ the textbook of Greuel, Lossen and Shustin \cite{MR2290112}*{Theorem 2.48}.

If $H$ has full rank $3$, as it is the case for the original Markov surface, then it follows directly from the holomorphic Morse lemma that the normal form in the origin is $x^2 + y^2 + z^2$, i.e.\ it has only an $A_1$-singularity.  
\end{proof}

\begin{lemma}
\label{lemma-Markov-D4}
The Markov-type surface given by 
\[
x^2 + y^2 + z^2 + xyz - 8x - 8y - 8z + 28 = 0
\]
has an isolated singularity of type $D_4$ in the point $(2,2,2)$.
\end{lemma}
\begin{proof}
Shifting all the coordinates by $-2$, we obtain the following equation:
\[
x^2 + y^2 + z^2 + 2xy + 2yz + 2zx + xyz = 0 \Longleftrightarrow (x + y + z)^2 + xyz = 0
\]
By a change of coordinates $u = x + y + z - \frac{1}{2} y z $, this is equivalent to
\[
u^2 + \frac{1}{4} y^2 z^2 - y^2 z - y z^2 = 0
\]
The Hessian has co-rank $2$ and the $3$-jet decomposes into three different linear factors: $-y z (y + z)$. 
By \cite{MR2290112}*{Theorem 2.51} this singularity is of type $D_4$.
\end{proof}

\begin{example}
The Markov-type surface given by 
\[
x^2 + y^2 + z^2 + xyz - 4x + 4 = 0
\]
has an isolated singularity of type $A_3$ in the point $(2,0,0)$:
Centering the coordinates in $(2,0,0)$ and diagonalizing the Hessian, we obtain the equation
\[
x^2 + y^2 + \frac{1}{4}(yx^2 - z^2y)
\]
Following the procedure in \cite{MR2290112}*{Theorem 2.47}, this can be brought into the form 
$x^2 + y^2 - z^4/2$ and is therefore an isolated singularity of type $A_3$.
\end{example}

\bigskip

We will now investigate the isolated surfaces singularities of type $A_k$ and $D_k$. We will determine the germs of vector fields in such a singularity. The lemmas and their proofs follow the same pattern for both types of singularities. 

\begin{lemma}
\label{lemma-Ak-sing}
We consider the surface $\widetilde{M}_k$ given by $x^2 + y^2 + z^{k+1} = 0$ with $k \in \NN$, i.e.\ the model surface for an isolated $A_k$-singularity.
Then the $\CC[\widetilde{M}_k]$-algebra of polynomial vector fields on $\widetilde{M}_k$ is given by
\begin{equation*}
\left( \CC[\widetilde{M}_k] \widetilde{V}^x + \CC[\widetilde{M}_k] \widetilde{V}^y +\CC[\widetilde{M}_k] \widetilde{V}^z \right)  \oplus \mathrm{span}_{\CC} \{1, z, z^2, \dots, z^{k-1} \} \cdot \widetilde{\Lambda}
\end{equation*}
where the vector fields $\widetilde{V}^x, \widetilde{V}^y, \widetilde{V}^z$ are Hamiltonian w.r.t.\ the induced holomorphic symplectic form on $\widetilde{M}_k \setminus \{0\}$, but $\widetilde{\Lambda}$ is not a holomorphic symplectic vector field.
\end{lemma}

\begin{proof}
The normal direction is given by 
\[
N = 2x dx + 2y dy + (k+1)z^{k} dz
\]
The following vector fields are then easily seen to be tangential:
\[
\begin{alignedat}{4}
\widetilde{V}^x &:= & & \phantom{+} (k+1)z^{k} \frac{\partial}{\partial y}& - 2y \frac{\partial}{\partial z}& \\
\widetilde{V}^y &:= & -(k+1)z^{k} \frac{\partial}{\partial x}& & + 2x \frac{\partial}{\partial z}& \\
\widetilde{V}^z &:= & 2y \frac{\partial}{\partial x}& - 2x \frac{\partial}{\partial y}& 
\\
\widetilde{\Lambda} &:=  &(k+1)x \frac{\partial}{\partial x} & + (k+1)y \frac{\partial}{\partial y} & + 2z \frac{\partial}{\partial z} &
\end{alignedat} 
\]

The volume form on $\widetilde{M}_k \setminus \{0\}$ is given by
\[
\widetilde{\omega} = \frac{dx \wedge dy}{(k+1) z^k} = \frac{dy \wedge dz}{2x} = \frac{dz \wedge dx}{2y}
\]
Any other volume form will be a product of $\widetilde{\omega}$ with a nowhere vanishing holomorphic function. 

A straightforward computation shows:
\[
i_{\widetilde{V}^x} \widetilde{\omega} = - dx, \quad
i_{\widetilde{V}^y} \widetilde{\omega} = - dy, \quad
i_{\widetilde{V}^z} \widetilde{\omega} = - dz
\]
Hence, the vector fields $\widetilde{V}^x, \widetilde{V}^y, \widetilde{V}^z$ are Hamiltonian, and in particular holomorphic symplectic or volume preserving. 

However, $\Lambda$ does not preserve $\widetilde{\omega}$:
\begin{align*}
i_{\widetilde{\Lambda}} \widetilde{\omega}  &= \frac{x dy - y dx}{z^k} \\
d i_{\widetilde{\Lambda}} \widetilde{\omega} &= \frac{2}{z^k} dx \wedge dy + \frac{k}{z^{k+1}}(x dy \wedge dz + y dz \wedge dx) \\
 & = \widetilde{\omega} \neq 0
\end{align*}

Note that $dz(\widetilde{V}^x) = -2y$, $dz(\widetilde{V}^y) = 2x$ and $dz(\widetilde{\Lambda}) = 2z$. Hence, for any tangential vector field $W$ which must necessarily vanish in the origin, there exist $f^x, f^y, f^\Lambda \in \CC[\widetilde{M}_k]$ such that
\[
dz(\underbrace{W - f^x \widetilde{V}_x - f^y \widetilde{V}_y - f^{\widetilde{\Lambda}} \widetilde{\Lambda}}_{=: W'}) = 0.
\]
Next, we can write $W' = a \frac{\partial}{\partial x} + b \frac{\partial}{\partial y}$ with the coefficients $a,b \in \CC[x,y,z]$ satisfying
\[
2x a + 2y b = q \cdot (x^2 + y^2 + z^{k+1})
\]
for some $q \in \CC[x,y,z]$. This requires that $x y \mid q$ as well as $x \mid b$ and $y \mid a$. Writing $q = 2 x y q'$, $a = y a'$ and $b' = x b'$, we obtain
\[
a' + b' = q' \cdot (x^2 + y^2 + z^{k+1}) \Longleftrightarrow a' = -b' \mod (x^2 + y^2 + z^{k+1})
\]
It follows that $W' = \frac{a'}{2} \widetilde{V}^z$.

Further, we observe that 
\begin{align*}
x \widetilde{\Lambda} &= z \widetilde{V}^y - \frac{k+1}{2} y \widetilde{V}^z \\
y \widetilde{\Lambda} &= \frac{k+1}{2} x \widetilde{V}^z -  z \widetilde{V}^x \\
z^{k} \widetilde{\Lambda} &= y \widetilde{V}^x - x \widetilde{V}^y
\end{align*}
Comparing the coefficients in front of $\frac{\partial}{\partial z}$, we also see that \[
\widetilde{\Lambda}, z \widetilde{\Lambda}, z^2 \widetilde{\Lambda}, \dots, z^{k-1} \widetilde{\Lambda} \notin \left( \CC[\widetilde{M}_k] \widetilde{V}^x + \CC[\widetilde{M}_k] \widetilde{V}^y +\CC[\widetilde{M}_k] \widetilde{V}^z \right)
. \qedhere \]
\end{proof}

\begin{lemma}
\label{lemma-Dk-sing}
We consider the surface $\widehat{M}_k$ given by $x(y^2 + x^{k-2}) + z^2 = 0$ with $k \in \NN, k \geq 4$, i.e.\ the model surface for an isolated $D_k$-singularity.
Then the $\CC[\widehat{M}_k]$-algebra of polynomial vector fields on $\widehat{M}_k$ is given by
\begin{equation*}
\left( \CC[\widehat{M}_k] \widehat{V}^x + \CC[\widehat{M}_k] \widehat{V}^y +\CC[\widehat{M}_k] \widehat{V}^z \right)  \oplus \mathrm{span}_{\CC} \{1, x, x^2, \dots, x^{k-2} \} \cdot \widehat{\Lambda} \oplus \CC \cdot \widehat{\Kappa}
\end{equation*}
where the vector fields $\widehat{V}^x, \widehat{V}^y, \widehat{V}^z$ are Hamiltonian w.r.t.\ the induced holomorphic symplectic form on $\widehat{M}_k \setminus \{0\}$, but $\widehat{\Lambda}$ and $\widehat{\Kappa}$ are not holomorphic symplectic vector fields. 

\end{lemma}
\begin{proof}
The normal direction is given by 
\[
N = (y^2 + (k-1)x^{k-2}) dx + 2xy dy + 2z dz
\]
The following vector fields are then easily seen to be tangential:
\[
\begin{alignedat}{4}
\widehat{V}^x &:= & & \phantom{+} 2z\frac{\partial}{\partial y}& - 2xy \frac{\partial}{\partial z}& \\
\widehat{V}^y &:= & -2z \frac{\partial}{\partial x}& & + (y^2+(k-1)x^{k-2}) \frac{\partial}{\partial z}& \\
\widehat{V}^z &:= & 2xy \frac{\partial}{\partial x}& - (y^2+(k-1)x^{k-2}) \frac{\partial}{\partial y}& 
\\
\widehat{\Lambda} &:=  &2x \frac{\partial}{\partial x} & + (k-2)y \frac{\partial}{\partial y} & + (k-1) z \frac{\partial}{\partial z} & \\
\widehat{\Kappa}  &:= & & \phantom{+} (y^2 + x^{k-2}) \frac{\partial}{\partial y}& + yz \frac{\partial}{\partial z}& 
\end{alignedat} 
\]

The volume form on $\widehat{M}_k \setminus \{0\}$ is given by
\[
\widehat{\omega} = \frac{dx \wedge dy}{2z} = \frac{dy \wedge dz}{y^2 + (k-1)x^{k-2}} = \frac{dz \wedge dx}{2xy}
\]

A straightforward computation shows:
\[
i_{\widehat{V}^x} \widehat{\omega} = - dx, \quad
i_{\widehat{V}^y} \widehat{\omega} = - dy, \quad
i_{\widehat{V}^z} \widehat{\omega} = - dz
\]

However, $\widehat{\Lambda}$ and $\widehat{\Kappa}$ do not preserve $\widehat{\omega}$:
\begin{align*}
i_{\widehat{\Lambda}} \widehat{\omega}   &= \frac{2x dy - (k-2)y dx}{2 z} \\
d i_{\widehat{\Lambda}} \widehat{\omega} &= \frac{k}{2z} dx \wedge dy + \frac{1}{2 z^2}(2x dy \wedge dz + (k-2) y dz \wedge dx) \\
                             &= \widehat{\omega} \neq 0 \\
i_{\widehat{\Kappa}} \widehat{\omega} &= \frac{-z}{2x} dx \\
d i_{\widehat{\Kappa}} \widehat{\omega}  &= \frac{-dz \wedge dx}{2x}
                             = y \widehat{\omega} \neq 0
\end{align*}
We observe that a vector field $W$ with $dx(W) = f^x(y)$ for a non-zero function $f^x(y) \in \CC[y]$ is never tangential: $N(W)$ contains a summand $y^2 f^x(y)$ which can never be compensated $\mod x(y^2 + x^{k-2}) + z^2$. 

Note that $dx(\widehat{V}^y) = -2z$, $dz(\widehat{V}^z) = 2xy$ and $dz(\widehat{\Lambda}) = 2x$. Hence, for any tangential vector field $W$ which must necessarily vanish in the origin, there exist $f^y, f^{\widehat{\Lambda}} \in \CC[\widehat{M}_k]$ such that
\[
dx(\underbrace{W - f^y \widehat{V}_y - f^{\widehat{\Lambda}} \widehat{\Lambda}}_{=: W'}) = 0.
\]
If $z$ divides any summand of $dy(W')$, then we consider $W'' = W' - f^{x} \widehat{V}^x$ for a suitable $f^x \in \CC[\widehat{M}_k]$ such that $z$ does not divide any summand of $dy(W'')$, i.e.\ we can
write $W'' = b \frac{\partial}{\partial y} + c \frac{\partial}{\partial z}$ with the coefficients $b \in \CC[x,y]$ and $c \in \CC[x,y,z]$ satisfying
\[
2xy b + 2z c = q \cdot (x y^2 + x^{k-1} + z^{2})
\]
for some $q \in \CC[x,y,z]$. This implies that
\[
z \mid (2xyb - q(xy^2 + x^{k+1})
\]
which is only possible if the term vanishes. As a consequence, we obtain $q = 2 q' y$ and $b = q' \cdot (y^2 + x^{k-1})$, and
\[
W'' = f^{\widehat{\Kappa}} \widehat{\Kappa}
\]
with $f^{\widehat{\Kappa}} = q' \in \CC[x,y]$.

Further, we observe that 
\begin{align*}
 z \widehat{\Lambda} &= -x \widehat{V}^y + \frac{k-2}{2} y \widehat{V}^x \\
 y \widehat{\Lambda} &=  -V^z + (k-1) K \\
xy \widehat{\Lambda} &=  x \widehat{V}^z - \frac{k-1}{2} z \widehat{V}^x \\
x^{k-1} \widehat{\Lambda} &= (-xy^2 - z^2) \widehat{\Lambda} \\
 x \widehat{\Kappa} &= \frac{1}{2} z \widehat{V}^x \\
 y \widehat{\Kappa} &= z \widehat{V}^y - y \widehat{V}^z - x^{k-2} \widehat{\Lambda} 
\end{align*}

Comparing the coefficients in front of $\frac{\partial}{\partial x}$ and $\frac{\partial}{\partial y}$, we also see that \[
\widehat{\Lambda}, x \widehat{\Lambda}, x^2 \widehat{\Lambda}, \dots, x^{k-2} \widehat{\Lambda}, \widehat{\Kappa} \notin \left( \CC[\widehat{M}_k] \widehat{V}^x + \CC[\widehat{M}_k] \widehat{V}^y +\CC[\widehat{M}_k] \widehat{V}^z \right)
. \qedhere \]
\end{proof}

\begin{proposition}
\label{prop-Markov-spanning}
Let $M$ be a normal affine complex surface with isolated singularities of type $A_k$ or $D_k$. Assume further that $M$ admits a holomorphic symplectic form on its regular part. Let $\Theta^j, j \in \{1, \dots, N\},$ be Hamiltonian vector fields on $M$ that span the tangent space in every regular point of $M$. Then every holomorphic symplectic vector field $\Xi$ on $M$ can be written as
\[
\Xi = \sum_{j=1}^{N} f_j \Theta^j
\]
where $f_j \colon M \to \CC$ are holomorphic functions.
\end{proposition}

\begin{corollary}
\label{cor-Markov-spanning}
Let $M$ be a surface of Markov-type with the natural symplectic form $\omega$ on its regular part. 
Then every holomorphic symplectic vector field $\Xi$ on $M$ can be written as
\begin{equation}
\label{eq-globalsectionsympl}
\Xi = f^x V^x + f^y V^y + f^z V^z
\end{equation}
for suitable holomorphic functions $f^x, f^y, f^z \colon M \to \CC$.
\end{corollary}
\begin{proof}
By Proposition \ref{prop-Markov-singularities} only isolated singularities of type $A_k$ and $D_4$ occur in a Markov-type surface. The Hamiltonian vector fields $V^x, V^y, V^z$ span the tangent spaces in the regular part of $M$ by Remark \ref{rem-spanning}.
\end{proof}

\begin{proof}[Proof of Proposition \ref{prop-Markov-spanning}]
Let $\omega$ be the natural volume form on $M_{\mathrm{reg}}$ and $s \in M_{\mathrm{sing}}$. Let $\widetilde{\omega}$ be the standard volume form on the model surface for the isolated singularity $s$, and let $\Phi$ be the locally biholomorphic change of coordinates of the ambient $\CC^3$ to this model surface $\widetilde{M}$ with $\Phi(s) = 0$. Then we have $\Phi_\ast \omega = h \cdot \widetilde{\omega}$ for some nowhere vanishing holomorphic function $h$ in a neighborhood of $0$. This function extends to the singularity $0$, and it must not vanish in $0$, for otherwise it would vanish on a curve on the surface through $0$.

Therefore, by Lemma \ref{lemma-Ak-sing} and Lemma \ref{lemma-Dk-sing}, the germ of a holomorphic vector field $V$ on $M$ in a singularity $s \in M$ contains a contribution of $\Phi^\ast \left( \mathrm{span}_{\CC} \{1, z, \dots, z^{k-1}\} \widetilde{\Lambda} \right)$ in the case of an $A_k$-singularity or a contribution of $\Phi^\ast \left( \mathrm{span}_{\CC} \{1, x, \dots, x^{k-2}\}\widehat{\Lambda} \oplus \CC \cdot \widehat{K} \right)$ in the case of a $D_k$-singularity if and only if $i_V \omega$ has a non-removable singularity in $s$:
\[
\Phi_\ast i_V \omega = i_{\Phi_\ast V} \Phi_\ast \omega = i_{\Phi_\ast V} h \widetilde{\omega} = h i_{\Phi_\ast V} \widetilde{\omega}
\]

\noindent\textbf{Claim:} Such a contribution can't be symplectic.

\noindent\textbf{Proof in case of an $A_k$-singularity:}
	
Let $\Phi^\ast W \in \Phi^\ast \left( \mathrm{span}_{\CC} \{1, z, \dots, z^{k-1}\} \widetilde{\Lambda} \right)$. Assume to get a contradiction that $\Phi^\ast W$ is symplectic, i.e. $0 = d i_{\Phi^\ast W} \omega$.
\[
\Longrightarrow 0 = \Phi_\ast(d i_{\Phi^\ast W} \omega) = d i_W \Phi_\ast \omega = d i_W (h \widetilde{\omega}) = d(fh) \wedge i_{\widetilde{\Lambda}} \widetilde{\omega} + fh \cdot \widetilde{\omega}
\]
where $0 \neq f(z) \in \CC[z]$ is of degree at most $k-1$ and such that $W = f(z) \widetilde{\Lambda}$. On the model surface for the $A_k$-singularity we can now write explicitly, after multiplication with $z^k/f(z)$:
\[
d(fh) \wedge \frac{x dy - y dx}{f} + h \cdot \frac{dx \wedge dy}{(k+1)} = 0
\]
We first take the limit $(x,y) \to (0,0)$ and then $z \to 0$ which yields $h(0)=0$, a contradiction!

\noindent\textbf{The proof in case of an $D_k$-singularity is similar:}

We have $W = f(x) \widehat{\Lambda} + \varepsilon \widehat{\Kappa}$
where $f(x) \in \CC[x]$ is a polynomial of degree at most $k-2$ and $\varepsilon \in \CC$ where at least $f \neq 0$ or $\varepsilon \neq 0$. 
	
On the model surface for the $D_k$-singularity we can write explicitly, after multiplication with $2z$:
\begin{align*}
d(f h) \wedge (2x dy - (k-2) y dx) + \varepsilon dh \wedge \frac{-z^2 dx}{2x} + h \cdot (f + \varepsilon y) dx \wedge dy = 0
\end{align*}
If $f(0) \neq 0$, we take first take the limit $z \to 0$, then $(x,y) \to (0,0)$ and obtain $h(0) = 0$, a contradiction!

Note that $df = f'(x) dx$, hence $df \wedge dx = 0$. If $f(0) = 0$ and $\varepsilon \neq 0$, we first take the limit $z \to 0$, then $x \to 0$ and obtain $h(0,y,0) \cdot y = 0$, but $h$ must not vanish for (small) $y \neq 0$, a contradiction!

The remaining case is $f(0) = 0$ and $\varepsilon = 0$. We write $f(x) = x^m f_1(x)$ with $f_1(0) \neq 0$ and divide the equation by $f$:
\[
(df / f \cdot h + dh) \wedge (2x dy - (k-2) y dx)  + h \cdot dx \wedge dy = 0
\] 
We first take the limit $(y,z) \to (0,0)$. 
Note that \[
\frac{\partial h}{\partial z} dz \wedge 2x dy = \frac{\partial h}{\partial z} xz \frac{y^2 + (k-1) x^{k-2}}{-xy^2-x^{k-1}} dx \wedge dy<
\] 
which already vanishes. This leads to the following ODE for $h$:
\[
(m + 1) \cdot h(x,0,0) + 2x \frac{\partial h(x,0,0)}{\partial x} = 0
\]
Unless $h$ vanishes identically, the solution of this ODE blows up in $0$, a contradiction! This proves the claim. 

\bigskip

Once we know that the inner product of the symplectic vector fields with the symplectic form extends to the singularities and that the Poincaré lemma (see Remark \ref{rem-Poincare} above) holds for the singularities we encounter here, we can proceed in the usual way:
The subsheaf $\mathcal{H}$ of holomorphic symplectic sections of the tangent sheaf $\mathcal{T}$ on $M$ is coherent, since all the other sheaves in the following exact sequence are coherent:
\[
\begin{tikzcd}
 0 \arrow[r] & \mathcal{H} \arrow[r] & \mathcal{T} \arrow[r, "V \mapsto i_V \omega"] & \Omega^1 \arrow[r, "d"] & \Omega^2 \arrow[r] & 0 
\end{tikzcd}
\]
By an application of the Nakayama lemma, the germs of the Hamiltonian vector fields $\Theta^j$ generate the stalk of the tangent sheaf in every point of the regular part of $M$. In particular, they also generate the stalk of $\mathcal{H}$ in every point of the regular part. From the above discussion it follows that the $\Theta^j$ also generate the stalk for $\mathcal{H}$ in every singular point. By an application of Cartan's Theorem B, every global section is of the form \eqref{eq-globalsectionsympl}.
\end{proof}

\begin{example}
For the original Markov surface
\[
x^2 + y^2 + z^2 = 3xyz
\]
the holomorphic linear combinations of $V^x, V^y, V^z$ only miss a one-dimensional subspace of germs in the origin which corresponds to $\CC \widetilde{\Lambda}$ in the model surface $\widetilde{M}_1$. In the Markov surface, a locally defined vector field corresponding to $\widetilde{\Lambda}$ (modulo $\CC[M] V^x + \CC[M] V^y + \CC[M] V^z$) is given by:
\begin{align*}
\Lambda &=
(x + \frac{1}{2}yz + \frac{3}{4}xy^2 + \dots) \frac{\partial}{\partial x} \\
&+ (y + \frac{1}{2}xz + \frac{3}{4}yz^2 + \dots) \frac{\partial}{\partial y}
+ (z + \frac{1}{2}xy + \frac{3}{4}zx^2 + \dots) \frac{\partial}{\partial z}
\end{align*}
However, $d i_\omega \Lambda = (-2 + \frac{4z}{2z - 3xy} - \frac{4z^2}{(2z - 3xy)^2} + \dots ) \omega \neq 0$.
\end{example}

\section{Proof of the Main Theorem}
\label{sec-main}

Recall that on the regular part of $M$, the holomorphic volume form is given by
\begin{equation*}
\omega = \frac{dx \wedge dy}{2z + Exy - C} = \frac{dy \wedge dz}{2x + Eyz - A} = \frac{dz \wedge dx}{2y + Exz - B}
\end{equation*}
Since $M$ is a surface, the holomorphic volume form $\omega$ is also a holomorphic symplectic form.

From Equation \eqref{eq-poisson-general} and the definition of the form $\omega$ we obtain the following defining relations for the Poisson brackets:
\begin{align*}
\po{x}{y} &= 2z + Exy - C\\
\po{y}{z} &= 2x + Eyz - A\\
\po{z}{x} &= 2y + Exz - B
\end{align*}

By Lemma \ref{lemma-completeness} and its corollary we know that the vector fields corresponding to the Hamiltonians $x^k, y^k, z^k$ are complete: $V^x$, $V^y$, and $V^z$ are the vector fields corresponding to the Hamiltonian functions $x$, $y$ and $z$, respectively. Hence we consider the Lie algebra generated these Hamiltonian functions on $M$:
\[
L := \lie(1, x^k, y^k, z^k \,:\, k \in \NN)
\]

\begin{lemma}
\label{lemma-step1}
We have
\[
x^k y, \, x^k z, \, x y^k, \, y^k z, \, x z^k, \, y z^k  \in L
\]
for all $k \in \NN$.
\end{lemma}
\begin{proof}
From the Poisson brackets $\po{x}{y}$, $\po{y}{z}$ and $\po{z}{x}$ we immediately obtain that $xy, yz, zx \in L$. This covers the case $k=1$. 
We proceed by induction on $k$ with the induction hypothesis that $x^k y, y^k z, z^k x, x^k z, y^k x, z^k y \in L$:
\begin{align*}
\{ x, x^k y \} 
 &= 2x^k z + E x^{k+1} y - C x^k \\
\{ y, y^k z \} 
 &= 2y^k x + E y^{k+1} z - A y^k \\ 
\{ z, z^k x \} 
 &= 2z^k y + E z^{k+1} x - B z^k \\ 
\{ x, x^k z \} 
 &= -2x^k y - E x^{k+1} z + B x^k \\ 
\{ y, y^k x \} 
 &= -2y^k z - E y^{k+1} x + C y^k \\ 
\{ z, z^k y \} 
 &= -2z^k x - E z^{k+1} y + A z^k 
\end{align*}
On each line, the left-hand side is contained in $L$ by the induction assumption, and two out of three summands on the right-hand side are contained in $L$ by the induction assumption as well. And thus we obtain the desired new terms (those with coefficient $\pm E$) in the induction step. 
\end{proof}

Next, we focus on monomials of degree $k+2$ involving a factor of $x^k$.

\begin{lemma}
\label{lemma-step2}
$x^{k}yz, \; x^{k}y^2, \; x^{k}z^2 \in L$ for all $k \in \NN$
\end{lemma}
\begin{proof}
First treat the case $k=1$. We have $xyz \in L$ by the defining equation of $M$, and $x y^2, x z^2 \in L$ by Lemma \ref{lemma-step1}. 

We will now treat the general case. The left-hand side of the following computation is contained in $L$ by Lemma \ref{lemma-step1}:
\begin{align*}
\po{x^{k-p}y}{x^pz} &= x^{k} \po{y}{z} + p x^{k-1} z \po{y}{x} + (k-p) y x^{k-1} \po{x}{z} \\
&= - 2(k-p) x^{k-1}y^2 - 2p x^{k-1} z^2 -(k-1) Ex^{k}yz \\ 
&\qquad +2x^{k+1} - Ax^{k}  + Cp x^{k-1} z + B(k-p) x^{k-1} y 
\end{align*}
Evaluating the right-hand side for $p=0$ and for $p=k$, respectively, and taking into account the known summands due to Lemma \ref{lemma-step1} we obtain also
\begin{align}
- 2k x^{k-1}y^2 -(k-1) Ex^{k}yz \in L \label{eq-xyz-1} \\
- 2k x^{k-1}z^2 -(k-1) Ex^{k}yz \in L \label{eq-xyz-2}
\end{align}
for every $k \in \NN$. 
Taking their difference and using the defining equation, this yields
\begin{align}
&L \ni -2kx^{k-1}y^2 + 2kx^{k-1}z^2 \nonumber \\
&= - 2kx^{k-1}y^2 + 2kx^{k-1}(-x^2 -y^2 -Exyz + Ax + By + Cz + D) \nonumber \\
&= -4kx^{k-1}y^2 - 2kEx^{k}yz \mod L \label{eq-xyz-3}
\end{align}
Combining the Equations \eqref{eq-xyz-1}, \eqref{eq-xyz-2} and \eqref{eq-xyz-3}, we finally obtain 
\[x^k y z, x^{k-1} y^2, x^{k-1} z^2 \in L. \qedhere\]
\end{proof}


Now, we find all the remaining monomials on $M$.

\begin{lemma}
\label{lemma-step3}
$x^{k}y^m \in L, x^{k}y^{m-1}z \in L$ for $k, m \in \NN$.
\end{lemma}

\begin{proof}
By induction on $m$:
Our induction hypothesis is
\[
\forall \, k \in \NN: \quad x^{k}y^m \in L, x^{k}y^{m-1}z \in L
\] We now compute
\begin{align*}
&\po{x^{k-p}y^m}{x^pz} \\ 
&= mx^{k}y^{m-1}\po{y}{z} + pmx^{k-1}y^{m-1}z\po{y}{x} + (k-p)x^{k-1}y^m\po{x}{z} \\
 &= mx^{k}y^{m-1}(2x + Eyz - A) - pmx^{k-1}y^{m-1}z(2z + Exy - C) \\ &\quad- 2(k-p)x^{k-1}y^{m+1} - (k-p)Ex^{k}y^mz + B(k-p)x^{k-1}y^m  \\
 &= 2mx^{k+1}y^{m-1} - Amx^{k}y^{m-1} + B(k-p)x^{k-1}y^m - 2(k-p)x^{k-1}y^{m+1}\\
 &\quad+ (m-pm-(k-p))Ex^{k}y^{m}z - 2pmx^{k-1}y^{m-1}z^2 + Cpmx^{k-1}y^{m-1}z 
\end{align*}
Evaluating for $p=0$ and for $p=k$, we obtain the following new terms:
\begin{align}
-2k x^{k-1} y^{m+1} + (m-k)E x^k y^m z \in L \label{eq-xyyz-1} \\
-2km x^{k-1} y^{m-1} z^2 + m(1-k) E x^k y^m z \in L \label{eq-xyyz-3}
\end{align}
Taking the linear combination $\frac{1-k}{2k} \eqref{eq-xyyz-1} - \frac{m-k}{2km} \eqref{eq-xyyz-3}$ and using the defining equation, this yields
\begin{align}
L \ni -(1-k) x^{k-1} y^{m+1} + (m-k) x^{k-1} y^{m-1} z^2 \nonumber  \\
= -(1-k) x^{k-1} y^{m+1} \nonumber \\ + (m-k) x^{k-1} y^{m-1} (-x^2 - y^2 - Exyz + Ax + By + Cz + D) \nonumber \\
= -(1 + m - 2k) x^{k-1} y^{m+1} - E(m-k) x^{k} y^{m} z \mod L \label{eq-xyyz-4}
\end{align}
The sum $\eqref{eq-xyyz-1} + \eqref{eq-xyyz-4}$ yields $-(1 + m) x^{k-1} y^{m+1} \in L$. Hence, we also obtain $x^k y^m z \in L$ for $k \neq  m$.

For $k = m$, evaluate again \eqref{eq-xyyz-3}:
\begin{align*}
L \ni -2m^2 x^{m-1} y^{m-1} z^2 + m(1-m) E x^m y^m z \\
= 2m^2 x^{m-1} y^{m-1} (x^2 + y^2 + Exyz - Ax - By - Cz -D) \\	 + m(1-m) E x^m y^m z \\
= 2m^2 x^{m-1} y^{m+1} + m(1+m) E x^{m} y^{m} z \mod L
\end{align*}
Since we already have $x^{m-1} y^{m+1} \in L$, we finally obtain $E x^{m} y^{m} z \in L$.
\end{proof}

We are now ready to prove the Hamiltonian density property and the symplectic/volume density property.

\begin{proof}[Proof of Theorem \ref{thm-main-1}]
Every polynomial function in $\CC[M]$ can be written as a finite sum of the form
\[
\sum_{k,m=0}^{\infty} a_{km} x^k y^m + \sum_{k,m=0}^{\infty} b_{km} x^k y^m z
\]
with $a_{km}, b_{km} \in \CC$.
By the Lemmas \ref{lemma-step1}, \ref{lemma-step2} and \ref{lemma-step3} we have \[ L = \lie(1, x^k, y^k, z^k) = \CC[M].\]
Thanks to Lemma \ref{lemma-completeness} and its corollary, all of the vector fields corresponding to the Hamiltonians $1, x^k, y^k, z^k$ are complete and extend to the singular part of $M$. Since $\CC[M]$ is dense in $\holo(M)$, this establishes the Hamiltonian density property for $M$.
\end{proof}

\begin{proof}[Proof of Theorem \ref{thm-main-2}]
Since $M$ is Stein, Cartan's Theorem B ensures that the coherent sheaves (see Grauert and Kerner \cite{MR0170354} for the singular situation) of holomorphic functions $\holo$ and of holomorphic $p$-forms $\Omega^p$ are acyclic. 
Consider the following sequence:
\[
0 \to \CC \to \mathcal{O} \to \Omega^1 \to \dots
\]
By Cantat and Loray \cite{MR2649343}*{Lemma 3.6} $M$ is simply connected. By a result of Reiffen \cite{MR0223599}, the Poincar{\'e} Lemma holds in the singular case if a certain local holomorphic contractibility condition is satisfied. For hypersurfaces, this condition is equivalent to the germ of the hypersurface being given by a so-called quasi-homogeneous polynomial thanks to a result of Saito \cite{MR0294699}. By Proposition \ref{prop-Markov-singularities}, $M$ only has only isolated singularities of type $A_k$ and $D_4$. ADE singularities are always quasi-homogeneous.

Therefore, the above sequence is exact and gives an acyclic resolution of the sheaf $\CC$. By the formal de Rham lemma, the first holomorphic de Rham cohomology group $H_d^1(M)$ is isomorphic to $H^1(X, \CC) = 0$. 
Finally, by Lemma \ref{lemma-deRhamtrivial} with $\Omega = X = M$ and $\ell = 1$, the notions of holomorphic symplectic vector field and of Hamiltonian vector field coincide on $M$. The choice of $\ell = 1$ is justified by the analysis of the germs of vector fields in the singularities that is summed up in Corollary \ref{cor-Markov-spanning}.
\end{proof}

\begin{proof}[Proof of Theorem \ref{thm-groupdesc}]
Let $F \colon M \to M$ be any holomorphic symplectic automorphism. If $F$ is in the path-connected component of the identity, then by definition there exists a $\cont^1$-smooth $\varphi_t \colon [0,1] \to \mathrm{Aut}_\omega(M)$ such that $\varphi_0 = \id$ and $\varphi_1 = F$. 
Since $H^1_d(M) = 0$, we can choose $\Omega = M$ in Theorem \ref{thm-AL}. Since the Hamiltonian functions $1, x^k, y^k, z^k$ $(k \in \NN)$ were the generators, the corresponding flows generate the identity component of the group. As the Hamiltonian $1$ corresponds to the identity map, it can be omitted. 
\end{proof}

%
%
%
%
%

\section{Tame Sets}
\label{sec-tame}

On $\CC^n, n \geq 2$, or more generally on a Stein manifold with the (volume, symplectic, Hamiltonian) density property, the group of holomorphic automorphisms acts $m$-transitively for any $m \in \NN$. However, one can't expect to prescribe automorphisms on any infinite closed discrete subsets. 

For this reason, tame sets in $\CC^n, n \geq 2,$ have been introduced by Rosay and Rudin \cite{MR0929658}. Any two tame sets in $\CC^n$ can be mapped onto each other by a holomorphic automorphism. On the other hand, Rosay and Rudin also showed that non-tame sets exist in $\CC^n$. In fact, a result proved later by Winkelmann implies that every Stein manifold contains non-tame infinite closed discrete subsets \cite{MR1827508}.

The notion of tameness was generalized independently in different ways to complex manifolds by Ugolini and the author \cite{MR3906368} and by Winkelmann \cite{MR3978035}. For $\CC^n, n \geq 2,$ all three notions agree. We adapt the definition from \cite{MR3906368} to the singular situation:

\begin{definition}
Let $X$ be a normal reduced complex space. Let $G$ be a group that acts on $X$ through holomorphic automorphisms. An infinite closed discrete subset $A \subset X_{\mathrm{reg}}$ is called \emph{$G$-tame} if for every injective self-map $\imath \colon A \to A$ there exists $F \in G$ such that $F|A = \imath$.
\end{definition}

\begin{remark}
As a direct consequence of the definition, any infinite subset $B \subseteq A$ of a tame set $A$ is again tame.
\end{remark}

\begin{proposition}[\cite{MR3906368}]
Let $X$ be a normal reduced Stein space with the Hamiltonian density property or the symplectic density property. Let $(a_j)_{j \in \NN} \subset X$ and $(b_j)_{j \in \NN} \subset X$ be tame closed discrete subsets. Then there exists a holomorphic automorphism $F \colon X \to X$ such that $F(a_j) = b_j$ for all $j \in \NN$. 
\end{proposition}

The proposition was stated in \cite{MR3906368} for Stein manifolds with the density property, but its proof carries over to the singular setting using Theorem \ref{thm-AL}, the singular version of the Anders\'en--Lempert.

It is in general a difficult question to establish the existence of tame sets in Stein manifolds. So far, this has been achieved for complex linear groups and their homogeneous space of dimension at least $2$, for the Koras--Russell cubic threefold and Danielewski surfaces, see \cites{MR3906368, tametwo}.

Our goal is to prove that the ordered Markov triples form a tame subset of the Markov surface. We first need three technical lemmas. 

\begin{lemma}
\label{lemma-equivalentconditions}
Let $X := \frac{2x - 3yz}{\gamma(z)}$ and $Y := \frac{2y - 3xz}{\gamma(z)}$ for the same choice of the root $\gamma(z) = \sqrt{4 - 9z^2}$. 
Assume $z \neq 0$ and $4 - 9z^2 \neq 0$. 
Then
\begin{enumerate}
\item $X \pm i y \neq 0$
\item $x \pm i Y \neq 0$
\item $\displaystyle \frac{ix - Y}{i x + Y} \neq \frac{X + iy}{iy - X}$
\end{enumerate}
on the Markov surface $x^2 + y^2 + z^2 = 3xyz$. 
\end{lemma}
\begin{proof}
We prove the statements by negation:
\begin{enumerate}
\item \(  
X \pm i y = 0 \Longrightarrow X^2 = -y^2 \Longrightarrow x^2 + y^2 = 3xyz \Longrightarrow z = 0
\)
\item \(
x \pm i Y = 0 \Longrightarrow Y^2 = -x^2 \Longrightarrow x^2 + y^2 = 3xyz \Longrightarrow z = 0
\)
\item \( \displaystyle
\frac{ix - Y}{i x + Y} = \frac{X + iy}{iy - X} \Longrightarrow 0 = x X + y Y \Longrightarrow x^2 + y^2 = 3xyz \Longrightarrow z = 0
\) \qedhere
\end{enumerate} 
\end{proof}

\begin{lemma}
\label{lemma-surjective}
Let $\varphi^z_t$ be the flow map of the vector field $V^z$ on the Markov surface $x^2 + y^2 + z^2 = 3xyz$. Then
\[
\CC \ni t \mapsto \varphi^z_t(x,y,z) \in \{ (u,v) \in \CC^2 \,:\, u^2 + v^2 + z^2 = 3uvz \}
\]
is surjective if $z \neq 0$ and $4 - 9z^2 \neq 0$.
\end{lemma}
\begin{proof}
The flow map $\varphi^z_t$ has been computed in Lemma \ref{lemma-completeness}. For a fixed choice of the root $\gamma(z)$, we write $T = \exp(i \gamma(z) t)$. The auxiliary map $\CC \ni t \mapsto T \in \CC^{\ast}$ is obviously surjective since $\gamma(z) \neq 0$ by assumption. Showing surjectivity is then equivalent to solving
\begin{equation}
\label{eq-surjective}
\left\{
\begin{array}{rcl}
u &=& x \cdot \displaystyle \frac{T + 1/T}{2} + Y \cdot \frac{T - 1/T}{2i} \\[12pt]
v &=& -X \cdot \displaystyle \frac{T - 1/T}{2i} + y \cdot \frac{T + 1/T}{2}
\end{array}
\right.
\end{equation}
The case $T = 0$ occurs only if the terms with $1/T$ cancel out, i.e.\ if $x + iY = 0$ and $X + iy = 0$. By Lemma \ref{lemma-equivalentconditions} this can't happen under our assumptions $z \neq 0$ and $4 - 9z^2 \neq 0$. 
We rewrite Equations \eqref{eq-surjective} as
\begin{equation*}
\left\{
\begin{array}{rcl}
0 &=& (ix + Y) T^2 - 2iu T + (ix - Y) \\
0 &=& (iy - X) T^2 - 2iv T + (X + iy)
\end{array}
\right.
\end{equation*}
The quadratic equation in $T$ has a non-zero solution for any $u$ since the zero-order term $ix - Y$ does not vanish. Moreover, also by Lemma \ref{lemma-equivalentconditions}, the equation can't be linear in $T$. Assume for the moment that the discriminant of this quadratic equation does not vanish, i.e.\ there are two distinct non-zero solutions for $T$. Since $v$ must satisfy the quadratic equation $u^2 + v^2 + z^2 = 3uvz$, there are only two possible choices for $v$. Again by Lemma \ref{lemma-equivalentconditions}, the zero-order terms of the normalized quadratic equations involving $u$ and $v$ are not equal. Hence, if the equation involving $u$ has two different solutions for $T$, then they necessarily correspond to two different values for $v$, and thus we obtain surjectivity.
If the discriminant of the quadratic equation involving $u$ vanishes, we reverse the roles of $u$ and $v$ in the preceding discussion. We are left with the case when both discriminants vanish simultaneously:
\begin{equation*}
\left\{
\begin{array}{rcl}
0 &=& -4u^2 - 4(ix + Y)(ix - Y) \\
0 &=& -4v^2 - 4(iy - X)(X + iy)
\end{array}
\right.
\end{equation*}
On the Markov surface $x^2 + y^2 + z^2 = 3xyz$ this is equivalent to
\[
u^2 = \frac{-4z^2}{4 - 9z^2} = v^2
\]
Finally, plugging this into the equation $u^2 + v^2 + z^2 = 3uvz$ yields again $z = 0$ or $4-9z^2 = 0$ which is excluded by assumption.
\end{proof}

\begin{lemma}
\label{lemma-faraway}
Let $z \in \CC$ with $z \neq 0$ and $4 - 9z^2 \neq 0$. Let $\pi^x, \pi^y \colon \CC^2 \to \CC$ denote the projections to the first and to the second coordinate, respectively. Then for any two compacts
\[
K, L \subset \{ (x,y) \in \CC^2 \,:\, x^2 + y^2 + z^2 = 3xyz \}
\]
there exists a non-empty open subset $U \subset \CC$ such that for all $t \in U$
\[
\varphi^z_t(K) \cap (\pi^x)^{-1}(L) = \emptyset \text{ and } \varphi^z_t(K) \cap (\pi^y)^{-1}(L) = \emptyset
\]
\end{lemma}
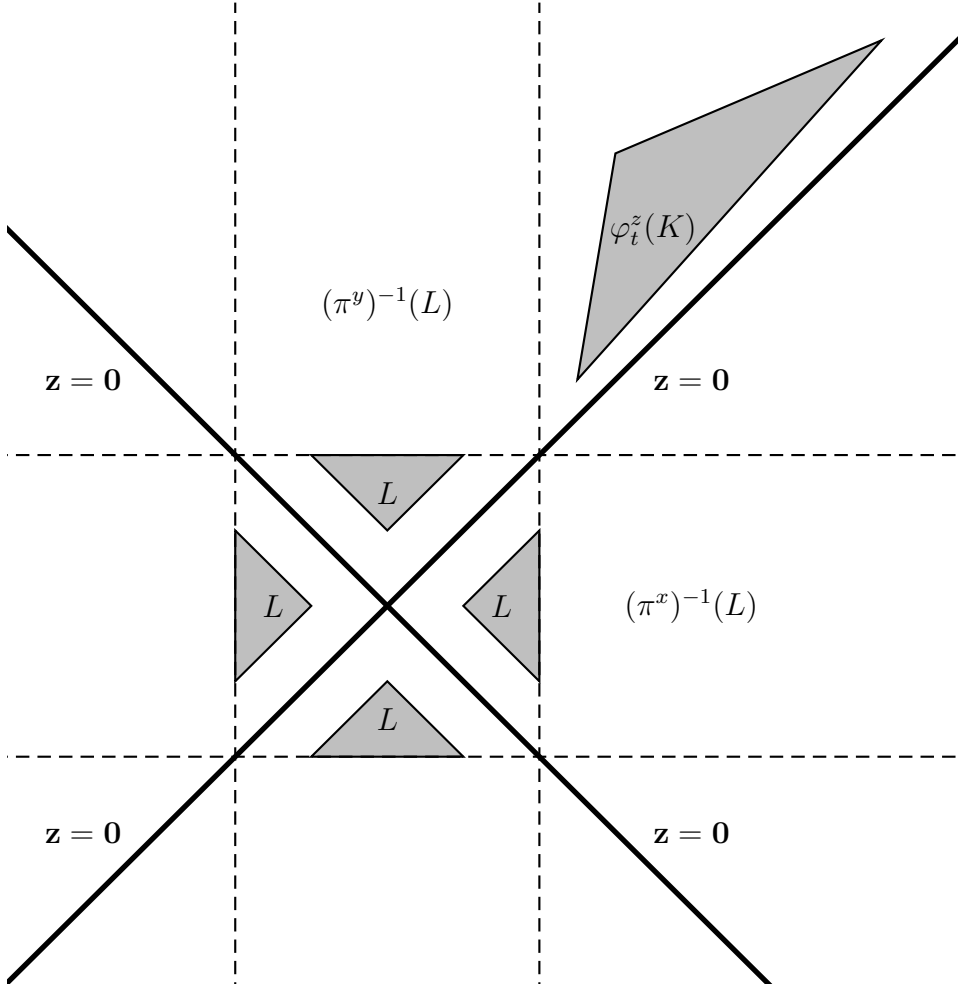
\begin{figure}[ht]
\begin{pspicture*}(-5,-5)(7.5,8)
\psline[linewidth=2pt](-8,-8)(8,8)
\psline[linewidth=2pt](-8,8)(8,-8)
\rput(-4,3){$\mathbf{z=0}$}
\rput(-4,-3){$\mathbf{z=0}$}
\rput(4,3){$\mathbf{z=0}$}
\rput(4,-3){$\mathbf{z=0}$}
\pspolygon[fillcolor=lightgray,fillstyle=solid](1,0)(2,1)(2,-1)
\pspolygon[fillcolor=lightgray,fillstyle=solid](-1,0)(-2,1)(-2,-1)
\pspolygon[fillcolor=lightgray,fillstyle=solid](0,1)(1,2)(-1,2)
\pspolygon[fillcolor=lightgray,fillstyle=solid](0,-1)(1,-2)(-1,-2)
\rput(1.5,0){$L$}
\rput(-1.5,0){$L$}
\rput(0,1.5){$L$}
\rput(0,-1.5){$L$}
\psline[linestyle=dashed](-8,2)(8,2)
\psline[linestyle=dashed](-8,-2)(8,-2)
\psline[linestyle=dashed](2,-8)(2,8)
\psline[linestyle=dashed](-2,-8)(-2,8)
\pspolygon[fillcolor=lightgray,fillstyle=solid](2.5,3)(3,6)(6.5,7.5)
\rput(3.5,5){$\varphi^z_t(K)$}
\rput(4,0){$(\pi^x)^{-1}(L)$}
\rput(0,4){$(\pi^y)^{-1}(L)$}
\end{pspicture*}
\caption{Illustration of Lemma \ref{lemma-faraway}}
\end{figure}

\begin{proof}
We choose a root $\gamma(z) = \sqrt{4 - 9z^2}$ for the whole calculation.  
We write $\gamma(z) t = \tau - i \sigma$ with $\sigma > 0$ and $\tau \in \RR$. For the $x$-coordinate and the $y$-coordinate of $\varphi^z_t$ we have
\begin{align*}
x(t) &= x \cos(\gamma(z) t) + Y \sin(\gamma(z) t) \\
     &= e^\sigma / 2 \cdot e^{i \tau} (x - i Y) + O(1) \\
y(t) &= -X \sin(\gamma(z) t) + y \cos(\gamma(z) t) \\
     &= e^\sigma / 2 \cdot e^{i \tau} (y + i X) + O(1) 
\end{align*}
By Lemma \ref{lemma-equivalentconditions} we have $x - i Y \neq 0$ and $y + iX \neq 0$. Since $K$ is compact, $|x - i Y| > 0$ and $|y + iX| > 0$ obtain minima $>0$ on $K$. We can now choose $\sigma > 0$ large enough to avoid all of the $x$-coordinates and $y$-coordinates of any given compact $L$. 
\end{proof}

\begin{remark}
\label{rem-finitelymany}
For every fixed Markov number $z$, there are infinitely many Markov triples of the form $(x,y,z)$ since
$(x,y,z) \mapsto (3xz-y,x,z)$ iteratively produces new solutions. On the other hand, if we restrict ourselves to ordered Markov triples, i.e.\ with $x \leq y \leq z$, only finitely many such Markov triples exist for a given $z$, see e.g.\ \cite{MR3098784}*{Theorem 3.19}. Of course, if the uniqueness conjecture were known to hold, then exactly one such Markov triple would exist. This would slightly simplify the first step of the following proof of Theorem \ref{theorem-tame}, and Lemma \ref{lemma-faraway} would not be needed anymore. 
\end{remark}

\begin{proof}[Proof of Theorem \ref{theorem-tame}]
Let $(p_j)_{j \in \NN} = (x_j,y_j,z_j)_{j \in \NN} \subset M \cap (\NN \times \NN \times \NN)$ be a sequence without repetition that enumerates the ordered Markov triples, and let $\eta \colon \NN \to \NN$ be any injective map. Our goal is to construct a holomorphic automorphism $F \colon M \to M$ such that $F(p_j) = p_{\eta(j)}$ for all $j \in \NN$. Denote by $\varphi_t^x, \varphi_t^y, \varphi_t^z$ the respective flows of the complete Hamiltonian vector fields $V^x, V^y, V^z$. We will construct $F$ as follows:
\[
F = G^{-1} \circ \varphi^x_{h(x)} \circ \varphi^y_{g(y)} \circ G
\]
with $G = \varphi^x_{f(z)}$, where $f$, $g$ and $h$ are holomorphic functions of the variables $z$, $y$ and $x$, respectively. Let $\pi^x, \pi^y, \pi^z \colon \CC^3 \to \CC$ denote the projections to the respective coordinate axes.
\begin{enumerate}
\item We choose a function $f$ such that $\varphi_{f(z)}^z(x_j,y_j,z_j) = (x'_j, y'_j, z_j)$ satisfies: 
$(x'_j)_{j \in \NN} \subset \CC \setminus \{0, \pm 2/3\}$ is a sequence without repetition and without accumulation points, and $(y'_j)_{j \in \NN} \subset \CC \setminus \{0, \pm 2/3\}$ is a sequence without repetition and without accumulation points: We construct $(x'_j, y'_j, z_j)$ inductively. For each $z_j = k \in \NN$ there are several, but finitely many choices for the pairs $(x_j, y_j, k)$, see Remark \ref{rem-finitelymany}. Let $K = \{(x_j, y_j) \in \CC^2 \,:\, z_j = k \}$ denote the the compact consisting of these points, and let $L = \{(x'_j, y'_j) \in \CC^2 \,:\, z_j < k \}$ denote the compact of all previously chosen images. By Lemma \ref{lemma-faraway} there exists $r_j \in \CC$ such that $\varphi_{r_j}^z(K) \cap L = \emptyset$ holds. Moreover, we can choose $r_j$ in an open set, and hence no repetition of coordinates will occur for a generic choice. By the Mittag-Leffler osculation theorem, there exists a holomorphic function $f \colon \CC \to \CC$ such that $f(z_j) = r_j$.
This ensures that the sequences $(x'_j)_{j \in \NN}$ and $(y'_j)_{j \in \NN}$ have the desired properties. 
\item Let $p'_j := G(p_j) = \varphi_{f(z)}^z(p_j)$. We choose a function $g$ such that
\[
\pi^x \circ \varphi^y_{g(y)}(p'_j) = \pi^x \circ p'_{\eta(j)}
\]
holds for all $j \in \NN$: By Lemma \ref{lemma-surjective} applied to the vector field $V^y$, for each $j \in \NN$ there exists $t_j \in \CC$ such that $\pi^x \circ \varphi^y_{t_j} (p'_j) = \pi^x \circ p'_{\eta(j)}$. By the Mittag-Leffler osculation theorem, there exists a holomorphic function $g \colon \CC \to \CC$ such that $g(y'_j) = t_j$. Note that all $y'_j$ are pairwise different.

\item We choose a function $h$ such that
\begin{align*}
\pi^y \circ \varphi^x_{h(x)} \circ \varphi^y_{g(y)}(p'_j) &= \pi^y \circ p'_{\eta(j)} \\
\pi^z \circ \varphi^x_{h(x)} \circ \varphi^y_{g(y)}(p'_j) &= \pi^z \circ p'_{\eta(j)}
\end{align*}
holds for all $j \in \NN$: For each $j \in \NN$ there exists $s_j \in \CC$ such that 
\begin{align*}
\pi^y \circ \varphi^x_{s_j} \circ \varphi^y_{g(y)}(p'_j) &= \pi^y \circ p'_{\eta(j)} \\
\pi^z \circ \varphi^x_{s_j} \circ \varphi^y_{g(y)}(p'_j) &= \pi^z \circ p'_{\eta(j)}
\end{align*}
by Lemma \ref{lemma-surjective} applied to the vector field $V^x$. By the Mittag-Leffler osculation theorem, there exists a holomorphic function $h \colon \CC \to \CC$ such that $h(x''_j) = s_j$ where $x''_j = \pi^x \circ \varphi^y_{g(y)}(p'_j) = \pi^x \circ p'_{\eta(j)}$. Since the points $p'_{\eta(j)}$ form a subset of the points $p'_j$, their $x$-coordinates are also pairwise different. 
\end{enumerate} 
Finally, $G^{-1}$ maps the points $p'_j$ to $p_j$ for all $j$.
\end{proof}

\begin{remark}
The same construction can be applied to find tame sets in any of the Markov-type surfaces. However, in general there won't exist any tame set whose coordinates are integers or rational numbers. It is enough to choose a sequence of points $p_j = (x_j, y_j, z_j)$ in $M$ such that every sequence of coordinates is without repetition, without accumulation points and avoids a finite set of curves (e.g.\ $z \neq 0$ and $4-9z^2 \neq 0$ in case of $A=B=C=D=0$ and $E=-3$).
\end{remark}

\section*{Funding}
The author was supported by the European Union (ERC Advanced grant HPDR, 101053085 to Franc Forstneri\v{c}) . 

\section*{Conflict of Interest}
The author has no relevant competing interest to disclose.

\end{document}